\documentclass[11pt]{article}
\usepackage{float}
\usepackage{graphicx}
\usepackage{amsmath,amsthm}
\usepackage{amssymb}
\usepackage{mathrsfs}
\usepackage{psfrag}

\restylefloat{figure}

\newtheorem{theorem}{\bf Theorem}
\newtheorem{assumption}{\bf Assumption}

\newtheorem{lemma}[theorem]{\bf Lemma}

\newcommand{\bepsilon}{\mbox{\boldmath{$\varepsilon$}}}
\def\be{\begin{equation}}
\def\ee{\end{equation}}
\def\ben{\begin{eqnarray}}
\def\een{\end{eqnarray}}

\newcommand{\bR}{\mathbb{R}}

\newcommand{\bH}{\mathbb{H}}
\newcommand{\bQ}{\mathbf{Q}}

\newcommand{\bL}{\mathbb{L}}

\newcommand{\bphi}{\mbox{\boldmath{$\phi$}}}
\newcommand{\bpsi}{\mbox{\boldmath{$\psi$}}}
\newcommand{\bmu}{\mbox{\boldmath{$\mu$}}}

\newcommand{\bzeta}{\mbox{\boldmath{$\zeta$}}}
\newcommand{\boeta}{\mbox{\boldmath{$\eta$}}}
\newcommand{\bPi}{\mbox{\boldmath{$\Pi$}}}
\newcommand{\cP}{\mathcal{P}}
\newcommand{\cT}{\mathcal{T}}
\newcommand{\cC}{\mathcal{C}}
\newcommand{\cS}{\mathcal{S}}
\newcommand{\cN}{\mathcal{N}}
\newcommand{\cI}{\mathcal{I}}
\newcommand{\bx}{\mathbf{x}}

\newcommand{\bu}{\mathbf{u}}
\newcommand{\bN}{\mathbb{N}}

\newcommand{\bV}{\mathbf{V}}
\newcommand{\bW}{\mathbf{W}}
\newcommand{\bM}{\mathbf{M}}

\newcommand{\tB}{\mathtt{B}}
\newcommand{\tK}{\mathtt{K}}
\newcommand{\tD}{\mathtt{D}}
\newcommand{\tA}{\mathtt{A}}

\newcommand{\tR}{\mathtt{R}}

\title{Two simple finite element methods for 
Reissner--Mindlin plates with clamped boundary condition}
\author{Bishnu P.~Lamichhane
\thanks{School of Mathematical \& Physical Sciences,
Mathematics Building - V127,
University of Newcastle,
University Drive,
Callaghan, NSW 2308, Australia, 
 {\tt Bishnu.Lamichhane@newcastle.edu.au}}}
\begin{document}
\maketitle

\begin{abstract}
We present two simple finite element methods for the 
discretization of Reissner--Mindlin plate equations with 
{\em clamped} boundary condition. 
These finite element methods are based on  
discrete Lagrange multiplier spaces from mortar finite element 
techniques. We prove optimal a priori error estimates 
for both methods.
\end{abstract}



 \section{Introduction}
There has been an extensive 
research effort to design  finite element methods 
for the Reissner--Mindlin plate equations over the last 
three decades. A finite element discretization of Reissner--Mindlin plate is 
a challenging task as a standard discretization locks when 
the plate thickness becomes too small. 
So the main difficulty is to avoid {\em locking} 
when the plate thickness becomes really small. 
There are now many {\em locking-free} finite element techniques with sound 
mathematical analysis for these equations 
\cite{AF89,BF91,AB93,CL95,Lov96,AF97,CS98,Bra96,FT00,Bra01,ACC02,Lo05}.
However, most of these finite 
element techniques are too complicated or expensive. 
In this paper, we present two very simple finite element methods 
for Reissner--Mindlin plate equations 
with {\em clamped} boundary condition. 
These finite element methods are based on a finite element 
method described in \cite{AB93} for Reissner--Mindlin plate equations
with {\em simply supported} boundary condition. 
We combine the idea of mortar finite elements with 
the finite element method proposed in \cite{AB93} 
to modify the discrete Lagrange multiplier space 
leading to optimal and efficient finite element schemes for 
Reissner--Mindlin plate equations with 
{\em clamped} boundary condition.  We propose two 
Lagrange multiplier spaces: one is based on 
a standard Lagrange multiplier space for the mortar 
finite element proposed in \cite{BD98}, and 
the other is based on a dual Lagrange multiplier space 
proposed in \cite{KLP01}.  The first one gives a continuous 
Lagrange multiplier, whereas the second one yields 
a discontinuous Lagrange multiplier. 
 The stability and optimal approximation properties are 
shown for both approaches.  
We note that the second 
approach with the discontinuous Lagrange multiplier space 
for Reissner--Mindlin plate equations
with {\em simply supported} boundary condition has not 
been presented before, where boundary modification is not necessary. 
However, we only focus on 
the {\em clamped} case as it is the most difficult case 
of the boundary condition in plate theory.  Moreover,
the second choice of the Lagrange multiplier space 
allows an efficient static condensation of 
the Lagrange multiplier leading to a positive definite 
system. Hence this approach is more efficient from 
the computational point of view.  We note that 
we use finite element spaces with equal dimension for 
the Lagrange multiplier and the rotation of the 
transverse normal vector.

The rest of the 
paper is planned as follows. 
The next section briefly recalls the Reissner--Mindlin plate 
equations in a modified form as given in \cite{AB93}. 
Section \ref{sec:fe} is the main part of the paper, where 
we describe our finite element methods and show the construction 
of discrete Lagrange multiplier spaces. Finally, 
a conclusion is drawn in the last section.

\section{A mixed formulation of Reissner--Mindlin plate}
Let $\Omega\subset \bR^2$ be a bounded region with 
polygonal  boundary. 
We need the following Sobolev spaces for the variational 
formulation of the Reissner--Mindlin plate with the plate 
thickness $t$: 
\[\bH^1(\Omega) = [H^1(\Omega)]^2,\quad 
\bH^1_0(\Omega) = [H^1_0(\Omega)]^2,\quad \text{and}\quad 
\bL^2(\Omega) = [L^2(\Omega)]^2.\]
We consider the following modified mixed formulation of 
Reissner--Mindlin plate with clamped boundary condition 
proposed in \cite{AB93}.  The mixed formulation 
is to find $(\bphi,u,\bzeta) \in 
\bH_0^1(\Omega)\times H^1_0(\Omega)\times \bL^2(\Omega)$ such that 
\begin{align*}
 a(\bphi,u;\bpsi,v)&+b(\bpsi,v;\bzeta)&=&\ell(v),\quad &(\bpsi,v) 
&\in &\bH_0^1(\Omega)\times H^1_0(\Omega), \\
b(\bphi,u;\boeta) &- \frac{t^2}{\lambda(1-t^2)} (\bzeta,\boeta)& =& 0, \quad 
&\boeta& \in& \bL^2(\Omega),
\end{align*} 
where $\lambda$ is a material constant depending on Young's modulus 
$E$ and Poisson ratio $\nu$, and 
\begin{eqnarray}
a(\bphi,u;\bpsi,v) &=& \int_{\Omega} \cC \bepsilon(\bphi):\bepsilon(\bpsi)\,d\bx +\lambda \int_{\Omega} (\bphi-\nabla u)\cdot(\bpsi-\nabla v)\,d\bx, \\
b(\bpsi,v;\boeta)& =& \int_{\Omega} (\bpsi-\nabla v)\cdot\boeta\,d\bx,\quad 
\ell(v) = \int_{\Omega} g\,v\,d\bx.
\end{eqnarray}
Here  $g$ is the body force, 
$u$ is the transverse displacement or normal deflection 
of the mid-plane section of $ \Omega$, $\bphi$ is the rotation of the 
transverse normal vector,  $\bzeta$ is the Lagrange 
multiplier, $\cC$ is the fourth order tensor, and 
$\bepsilon(\bphi)$ is the symmetric part of the gradient of $\bphi$. 
In fact, $\bzeta$ is the scaled shear 
stress defined by 
\[ \bzeta = \frac{\lambda (1-t^2)} {t^{2}} \left(\bphi - \nabla u\right).\]

\section{Finite element discretization}\label{sec:fe}
We consider a quasi-uniform triangulation $\cT_h$ of the 
polygonal domain $\Omega$ with mesh-size $h$, where $\cT_h$
consists of triangles. 
Now we introduce the standard linear finite element
space $K_h\subset H^1(\Omega)$ defined on the triangulation $\cT_h$ 
\[K_h := \{v \in H^1(\Omega) :\, v_{|_{T}} \in \cP_1(T),\; T \in \cT_h\},
\]
and the space of bubble functions 
\[
B_h := \{b_T\in \cP_3(T):\, {b_T}_{|_{\partial T}}=0,\;\text{and}\; 
\int_{T}b_T\,d\bx >0,\; T \in \cT_h\},
\]
where $\cP_n(T)$ is the space of polynomials of degree 
$n$ in $T$ for $ n \in \bN$. 
The bubble function on an element $T$ can be defined as 
$b_T(x)=c_b \Pi_{i=1}^{3}\lambda_{T^i}(x),$
where $\lambda_{T^i}(x)$ are the barycentric coordinates of the element $T$
associated with 
vertices $x_{T^i}$ of $T$, $i=1,\cdots,3$,
and the constant $c_b$ is chosen in such a way that 
the value of $b_T$ at the barycenter of $T$ is one.  

Let  $S_h = H_0^1(\Omega) \cap K_h$. 
A finite element method for the {\em simply supported}
Reissner--Mindlin plate is proposed in \cite{AB93} using 
the finite element spaces   $W_h:=S_h\oplus B_h$ to discretize 
the  transverse displacement,  $\bV_h:=[S_h]^2$ to 
 discretize the rotation, and $\bM_h:=[K_h]^2$ to 
discretize the Lagrange multiplier space. 
This is the lowest order case for the transverse 
displacement and rotation using 
the continuous piecewise linear shear approximation in \cite{AB93}. 
Hence the discretization uses 
 equal order interpolation for the rotation and the transversal 
displacement, and is one of the simplest finite element methods. 
However, for the {\em clamped} boundary condition we need to have 
$\bV_h \subset  \bH_0^1(\Omega)$, and hence the stability 
condition is violated if we use $\bM_h=[K_h]^2$ to discretize the Lagrange
multiplier space, and if we use $\bM_h=[S_h]^2$ to discretize 
the Lagrange multiplier space, the approximation property of 
the scheme is lost as the the Lagrange multiplier $\bzeta$ 
is not assumed to satisfy the zero boundary condition. Indeed, 
 the discrete space $\bM_h\subset \bL^2(\Omega)$ for 
the Lagrange multiplier space should have the 
following approximation property
\[
\inf_{\bmu_h \in \bM_h}\|\bphi-\bmu_h\|_{L^2(\Omega)}\leq 
Ch |\bphi|_{1,\Omega},\quad \bphi \in \bH^1(\Omega).
\]

Our goal in this paper is to introduce two discrete spaces 
for the Lagrange multiplier space so that the resulting 
scheme is stable and has the optimal approximation 
property for the {\em clamped} plate. We also introduce 
a scheme where the Lagrange multiplier can be statically 
condensed out from the system leading to a positive definite 
formulation.

We now start with  finite element spaces for the 
transverse displacement $u$ and the rotation  $\bphi$ as
$W_h:=S_h \oplus B_h$ and $\bV_h := [S_h]^2$,  respectively.   
Let \[ 
\{\varphi_1,\varphi_2,\cdots,\varphi_m,\varphi_{m+1},\cdots, \varphi_n\}\] 
be the standard finite element basis for $K_h$, where 
$n>m$ and  $\{\varphi_1,\varphi_2,\cdots,\varphi_m\}$
is a basis of $S_h$.   Note that the basis functions 
$\{\varphi_{m+1},\cdots, \varphi_n\}$ are associated with 
the boundary.  We use the idea of 
boundary modification of 
Lagrange multiplier spaces in mortar finite element methods 
\cite{BMP93,KLP01,Lam06} to 
construct a discrete Lagrange multiplier space for 
Reissner--Mindlin plate equations. 

Let $M_h\subset L^2(\Omega)$ be a piecewise 
polynomial space with respect to the mesh $\cT_h$  
to be specified later which satisfies the 
following assumptions:
\begin{assumption}\label{A1A2}
\begin{itemize} 
\item[\ref{A1A2}(i)] $\dim M_h = \dim S_h$.
\item[\ref{A1A2}(ii)] There is a constant $\beta>0$ independent of 
the triangulation $\cT_h$ such that 
\begin{eqnarray}
\|\phi_h\|_{L^2(\Omega)} \leq \beta \sup_{\mu_h \in M_h \backslash\{0\}} 
\frac{\int_{\Omega} \mu_h\phi_h\,d\bx} {\|\mu_h\|_{L^2(\Omega)}},
\quad \phi_h \in S_h.
\end{eqnarray}
\item[\ref{A1A2}(iii)] The space $M_h$ has the approximation property:
\begin{equation}
\inf_{\mu_h \in M_h}\|\mu-\mu_h\|_{L^2(\Omega)}\leq 
Ch |\mu|_{H^1(\Omega)},\quad \mu \in H^1(\Omega).
\end{equation}
\item[\ref{A1A2}(iv)]
There exist two bounded linear projectors  
$\bQ_h :\bH_0^1(\Omega) \rightarrow \bV_h$ 
and 
$R_h : H_0^1(\Omega) \rightarrow W_h$ 
for which 
\[ b( Q_h\bpsi,R_h v;\boeta_h) = b(\bpsi,v;\boeta_h),\quad 
\boeta_h \in [M_h]^2.\]
\end{itemize}
\end{assumption} 
If these assumptions are satisfied, we obtain 
an optimal error estimate for the finite element approximation, 
see \cite{AB93}. 
Then the discrete space for the Lagrange multiplier space is 
defined as 
\[\bM_h= [M_h]^2\subset \bL^2(\Omega), \]
and the discrete saddle point formulation is to 
find $(\bphi_h,u_h,\bzeta_h) \in \bV_h \times W_h \times \bM_h$ 
such that  
\begin{equation}\label{dsaddle}
\begin{array}{ccccccc}
a(\bphi_h,u_h;\bpsi_h,v_h)&+
b(\bpsi_h,v_h;\bzeta_h)&=&(g,v_h),\quad &(\bpsi_h,v_h) 
&\in &\bV_h \times S_h, \\
b(\bphi_h,u_h;\boeta_h) &- \frac{t^2}{\lambda(1-t^2)} (\bzeta_h,\boeta_h)& =& 0, \quad 
&\boeta_h& \in& \bM_h.
\end{array}
\end{equation}

We now show two examples of discrete Lagrange multiplier spaces satisfying
above properties. The first example is based on the standard Lagrange
multiplier space for three-dimensional mortar finite elements proposed
in \cite{BD98}, and the second example is based on a dual Lagrange 
multiplier space proposed in \cite{KLP01}. As the dual Lagrange
multiplier space satisfies a biorthogonality relation with the finite
element space $S_h$ leading to a diagonal Gram matrix, it allows 
an efficient solution technique. In fact, the Lagrange multiplier 
can be statically condensed out from the global system leading to
a reduced linear system in this case. This reduced linear system can
be solved more efficiently than the global saddle point system. The
Lagrange multiplier can easily be recovered 
just by inverting a diagonal matrix. 
One important factor in the construction of a Lagrange multiplier
space for the {\em clamped} boundary condition case is the boundary
modification so that Assumptions 1(i)--(iii) are satisfied. In order
to satisfy Assumption 1(iii), the discrete Lagrange multiplier space
should contain constants in $\Omega$. Therefore, it is not possible to
take $M_h =S_h$. Here we follow closely  \cite{BD98} 
for the construction and boundary modification of 
the discrete Lagrange multiplier space.

\subsection{Standard Lagrange multiplier space $\bM_h^1$}
In the following, we assume that each triangle has at least 
one interior vertex. A necessary modification 
for the case where a triangle has all its vertices on the boundary 
is given in \cite{BD98}. 

Let $\cN$, $\cN_0$ and $\partial\cN$ be the set of all vertices 
of $\cT_h$, the vertices of $\cT_h$ interior to $\Omega$, and 
the vertices of $\cT_h$ on the boundary of $\Omega$, 
respectively.  We define the set of all 
vertices which share a common edge with the vertex  $ i \in \cN$ as 
\[ \cS_i = \{j: \text{$i$ and $j$ share a common edge}\},\]
and the set of neighbouring vertices of $i \in \cN_0$ as 
\[ \cI_i =\{ j \in \cN_0:\,j \in \cS_i\}.\]

Then the set of all those interior vertices which have a neighbour on the 
boundary of $\Omega$ is defined as 
\[ \cI = \bigcup_{i \in \partial\cN}\cI_i,\qquad \text{See Figure \ref{SI}}.
\] 
 \begin{figure}[!ht]
\begin{center}
 \includegraphics[width =0.68\textwidth]{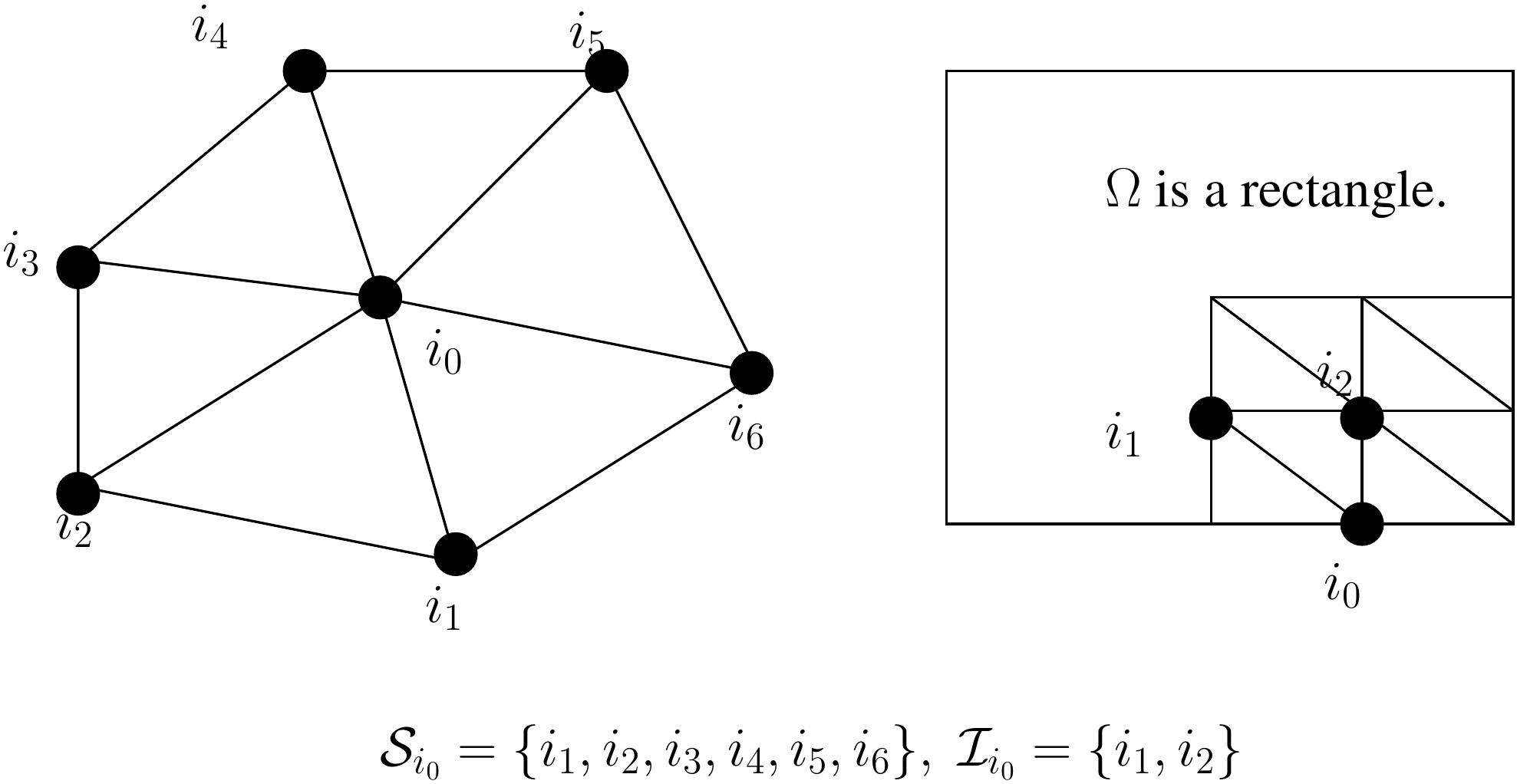}
 \caption{Examples for ${\mathcal S}_i$ and ${\mathcal I}_i$} 
\end{center}
 \label{SI}
\end{figure}

The finite element basis functions $\{\phi_1,\phi_2,\cdots,\phi_m\}$ 
for $M^1_h$ are defined as 
\[ \phi_i = \begin{cases} \varphi_i, & i \in \cN_0\backslash \cI\\
 \varphi_i + \sum_{j \in \partial \cN \cap \cS_i}A_{j,i} \varphi_j,\; 
A_{j,i} \geq 0, & i \in \cI
\end{cases}.
\]
We can immediately see that all the basis functions of $M^1_h$ are 
continuous, and $\dim M^1_h = \dim S_h$. Moreover, 
if the coefficients $A_{i,j}$ are 
chosen to satisfy 
\[ \sum_{j \in \cS_i} A_{i,j} =1,\quad i \in \cI,\]
Assumptions 1(ii) and 1(iii) are also satisfied,
see  \cite{BD98} for a proof. 
The vector Lagrange multiplier space is defined as $\bM^1_h =[M_h^1]^2$.

\begin{lemma}\label{lemma0}
There exist two bounded linear projectors  
$\bQ_h :\bH_0^1(\Omega) \rightarrow \bV_h$ 
and 
$R_h : H_0^1(\Omega) \rightarrow W_h$ 
for which 
\[ b( Q_h\bpsi,R_h v;\boeta_h) = b(\bpsi,v;\boeta_h),\quad 
\boeta_h \in \bM^1_h.\]
\end{lemma}
\begin{proof}
 We first define two operators 
 $Z_h : H_0^1(\Omega) \rightarrow B_h$ and 
$Q_h :H_0^1(\Omega) \rightarrow S_h$ as 
\[ \int_T (v- Z_h v) \, d\bx =0,\quad T \in \cT_h,\]
 and \[ \int_{\Omega} (v-Q_h v)\, \eta_h\,d\bx =0,\quad 
\eta_h \in M^1_h,\]
respectively.  The first operator $Z_h$ is well-defined as we can 
have a bubble function 
$b_T \in B_h$ with 
\[ \int_T b_T \,d\bx = \int_T v  \, d\bx.\]
The second operator $Q_h$ is well-defined and stable in 
$L^2$ and $H^1$ -norms 
due to Assumption 1(i)-(ii), see, e.g., \cite{BD98,KLP01}. Let $P_h$ be the $L^2$-orthogonal projection onto 
$S_h$. Now we define the operator  $R_h : H^1_0(\Omega) \rightarrow W_h$ 
with 
\begin{equation}\label{fortin}
 R_h v= P_h v + Z_h (v-P_hv).
\end{equation}
Note that $R_h$ is a bounded linear projector onto $W_h=S_h\oplus B_h$ 
with the property \cite{ABF84,AB93,Bra01}
\[ \int_T (R_hv -v)\,d\bx =0,\quad T \in \cT_h.\] 
Let $\bQ_h:\bH_0^1(\Omega) \rightarrow \bV_h$  be the vector 
version of $Q_h$. That means $\bQ_h\bu = (Q_hu_1,Q_hu_2)$ for 
$\bu =(u_1,u_2) \in \bH_0^1(\Omega)$.
Then 
\begin{eqnarray*}
b(\bQ_h\bpsi,R_hv;\boeta_h) & =&  
\int_{\Omega} \left(\bQ_h\bpsi - \nabla R_hv\right) \cdot 
\boeta_h \,d\bx \\ &=& \int_{\Omega} \bQ_h\bpsi\cdot \boeta_h\,d\bx 
- \int_{\Omega}   \nabla R_hv\cdot \boeta_h \,d\bx\\
&=& \int_{\Omega} \bpsi\cdot \boeta_h\,d\bx 
+ \int_{\Omega}   R_hv \,\nabla\cdot \boeta_h \,d\bx,
\end{eqnarray*}
where we use the divergence theorem and the fact that 
$\boeta_h$ is continuous. 
Since $\boeta_h$ is a continuous function 
and $\nabla\cdot\boeta_h$ is a piecewise constant with 
respect to the mesh $\cT_h$, we have 
\begin{eqnarray*}
 b(\bQ_h\bpsi,R_hv;\boeta_h) & =&   \int_{\Omega} \bpsi\cdot \boeta_h\,d\bx 
+ \int_{\Omega}   v \,\nabla\cdot \boeta_h \,d\bx \\
&=& \int_{\Omega} (\bpsi - \nabla v)\cdot \boeta_h\,d\bx = 
 b(\bpsi,v;\boeta_h).
\end{eqnarray*}
The boundedness of $R_h$ in $H^1$-norm can be shown as in 
\cite{ABF84}. 
\end{proof}

Thus we have the following theorem from 
the theory of saddle point problem \cite{BF91,Bra01}.

\begin{theorem}\label{th0}
There exists a constant $C$ independent of $t$ and $h$ such that 
\begin{eqnarray*}
 \|\bphi-\bphi_h\|_{H^1(\Omega)} + \|u - u_h \|_{H^1(\Omega)} 
+ \||\bzeta -\bzeta_h\||_t  \leq \\
C \left( \inf_{\bpsi_h \in \bV_h} 
\|\bphi-\bpsi_h\|_{H^1(\Omega)} + \inf_{v_h \in W_h} 
\|u-v_h\|_{H^1(\Omega)}  +  \inf_{\boeta_h \in \bM^1_h} 
 \||\bzeta -\boeta_h\||_t  \right),
\end{eqnarray*}
where the norm $\||\cdot\||_t$ is defined as 
\[  \||\bzeta\||_t = \|\bzeta\|_{H^{-1}(\Omega)} + \|\nabla\cdot\bzeta \|_{H^{-1}(\Omega)}  
+ t \|\bzeta \|_{L^2(\Omega}.\]
Moreover, if $\bphi \in \bH^2(\Omega)$, 
$u \in H^2(\Omega)$ and $\bzeta\in \bH^1(\Omega)$, 
the approximation properties of $\bV_h$, $W_h$ and 
$\bM^1_h$ imply that 
\[ \|\bphi-\bphi_h\|_{H^1(\Omega)} + \|u - u_h \|_{H^1(\Omega)} 
+ \||\bzeta -\bzeta_h\||_t  \leq \\
C h \left(\|\bphi\|_{H^2(\Omega)} + \|u\|_{H^2(\Omega)}  + \|\bzeta\|_{H^1(\Omega)}\right).\]
\end{theorem}

\subsection{Dual Lagrange multiplier space $\bM_h^2$}
Another possibility of a discrete Lagrange multiplier space is 
to consider the  dual Lagrange multiplier space proposed 
for mortar finite elements in \cite{KLP01,BWHabil}. Interestingly, 
the boundary modification can be exactly done as in the case 
of standard Lagrange multiplier space. We start with the 
basis for the dual Lagrange multiplier space $\tilde{M}_h$
including the degree of freedom on the boundary of $\Omega$. 
Let $\{\tilde{\mu}_1,\tilde{\mu}_2,\cdots,\tilde{\mu}_m,\tilde{\mu}_{m+1},\cdots, \tilde{\mu}_n\}$ 
be the basis for $\tilde{M}_h$, which is biorthogonal to the basis 
 $\{\varphi_1,\varphi_2,\cdots,\varphi_m,\varphi_{m+1},\cdots, \varphi_n\}$ 
of $K_h$ so that these basis functions satisfy 
the  biorthogonality relation
\begin{eqnarray} \label{biorth}
  \int_{\Omega} \tilde{\mu}_i \ \varphi_j \, d\bx = c_j \delta_{ij},
\; c_j\neq 0,\; 1\le i,j \le n,
\end{eqnarray}
 where $n := \dim \tilde{M}_h = \dim K_h$,  $\delta_{ij}$ is 
 the Kronecker symbol, and $c_j$ a scaling factor. 
In fact, we can construct local basis functions 
for $\tilde{M}_h$ on the reference triangle $\hat T$ 
so that the global basis functions for $\tilde{M}_h$
 are constructed by  gluing these local basis functions together. 
This means that we can use a standard assembling routine for the functions 
in  $\tilde{M}_h$, and the basis functions of $\tilde M_h$ are also 
associated with the finite element vertices as the basis functions of $K_h$. 
For the reference triangle $\hat T:=\{(x,y) :\, 0\leq x,0\leq y,x+y\leq 1\}$, we have 
\begin{eqnarray*}
  \hat \mu_1:=3-4x-4y,\,
  \hat\mu_2:=4x-1,\;\text{and}\;
  \hat\mu_3:=4y-1.
\end{eqnarray*}

The finite element basis functions $\{\mu_1,\mu_2,\cdots,\mu_m\}$ 
for $M^2_h$ are defined as 
\[ \mu_i = \begin{cases} \tilde{\mu}_i, & i \in \cN_0\backslash \cI\\
 \tilde{\mu}_i + \sum_{j \in \partial \cN \cap \cS_i}A_{j,i} \tilde{\mu}_j,\, 
A_{j,i}\geq 0,\,& i \in \cI
\end{cases}.
\]
 
If the coefficients $A_{i,j}$ are  chosen to satisfy 
\[ \sum_{j \in \cS_i} A_{i,j} =1,\quad i \in \cI,\]
Assumptions 1(ii) and 1(iii) are also satisfied 
exactly as in the case of the standard Lagrange multiplier space.
The proof of Assumption 1(ii) is much easier than in \cite{BD98} 
due to the biorthogonality relation. 
In fact, if we set $\phi_h=\sum_{k=1}^{n} a_k \phi_k \in S_h$ and 
$\mu_h=\sum_{k=1}^{n} a_k \mu_k \in M^2_h$, 
the biorthogonality relation \eqref{biorth}
and the quasi-uniformity assumption imply that 
\begin{eqnarray*} 
\int_{\Omega} \varphi_h \mu_h \,d\bx =\sum_{i,j=1}^{n} a_i a_j\int_{\Omega}\varphi_i\,\mu_j\,d\bx
=\sum_{i=1}^{n} a_i^2 c_i 
\geq C \sum_{i=1}^{n} a_i^2  h_{i}^2\geq C\|\varphi_h\|^2_{L^2(\Omega)},
 \end{eqnarray*}
where $h_i$ denotes the mesh-size at $i$the vertex.
Taking into account the fact that $\|\varphi_h\|^2_{L^2(\Omega)} \equiv 
\|\mu_h\|^2_{L^2(\Omega)} \equiv \sum_{i=1}^{n} a_i^2  h_{i}^d$,
we find that Assumption \ref{A1A2}(ii) is satisfied.
Since the sum of the local basis functions of $M^2_h$ is one, Assumption 
\ref{A1A2}(iii) can be proved as in \cite{KLP01}.  
The vector Lagrange multiplier space 
is defined as before $\bM_h^2 = [M_h^2]^2.$
Although the condition $\dim M^2_h = \dim S_h$ is satisfied as 
before, the basis functions for $M^2_h$ are not continuous, and 
we cannot show the existence of a bounded linear operator 
$R_h$ as in Lemma \ref{lemma0}. We need to use an 
alternative method.  As before we need to prove the following theorem to 
show the well-posedness of the discrete problem: 
\begin{theorem}\label{th1}
There exist two bounded linear projectors 
$\bQ_h :\bH_0^1(\Omega) \rightarrow \bV_h$ 
and 
$R_h : H_0^1(\Omega) \rightarrow W_h$ 
for which 
\[ b( \bQ_h\bpsi,R_h v;\boeta_h) = b(\bpsi,v;\boeta_h),\quad 
\boeta_h \in \bM^2_h.\]
\end{theorem}
We define the  projector $Q_h :H_0^1(\Omega) \rightarrow S_h$ as 
\[  \int_{\Omega} (v-Q_h v)\, \eta_h\,d\bx =0,\quad 
\eta_h \in M^2_h.\] 
Here $Q_h$ is well-defined and bounded in $L^2$ and 
$H^1$- norms  due 
to Assumption 1(i)-(ii), see \cite{Lam06}.  Our task is now to 
show the existence of the operator 
$R_h : H^1_0(\Omega) \rightarrow W_h$ satisfying 
\[ \int_{\Omega} \nabla R_h v\cdot \boeta_h\, d\bx  = 
\int_{\Omega} \nabla v\cdot \boeta_h\, d\bx,\;\boeta_h \in \bM^2_h.\]
In order to show the existence of this operator, we 
use the following result proved in \cite{Lam08}: 
\begin{lemma}\label{lemma1}
Let $\bW_h = [W_h]^2$, and 
\[ \tilde M_h^0 =  \left\{\mu_h = \sum_{i=1}^n a_i \tilde{\mu}_i, \, 
\int_{\Omega} \mu_h \,d\bx =0.\right\}\]
Then there exists a constant $\tilde \beta>0$ independent of mesh-size $h$ such that 
\begin{eqnarray}\label{is1} 
\sup_{\bu_h \in \bW_h}
\frac{\int_{\Omega}\nabla \cdot \bu_h \mu_h\,d\bx}{\|u_h\|_{H^1(\Omega)} } 
\geq\tilde \beta \|\mu_h\|_{L^2(\Omega)}\, ,\quad \mu_h \in \tilde M^0_h\, . 
\end{eqnarray}
\end{lemma}
Let \[L^2_0(\Omega) =\left\{ u \in L^2(\Omega): \, 
\int_{\Omega} u \, d\bx =0\right\}.\] 
Using Lemma \ref{lemma1} and the fact that 
the two spaces $\bH^1_0(\Omega)$ and $L^2_0(\Omega)$ satisfy the 
inf-sup condition 
\[ \sup_{\bu \in \bH^1_0(\Omega)}
\frac{\int_{\Omega}\nabla \cdot \bu \, \mu\,d\bx}{\|u\|_{H^1(\Omega)}, } 
\geq \beta \|\mu\|_{L^2(\Omega)} \]
we can show the existence of 
a bounded linear projector  $\bPi_h : \bH^1_0(\Omega) \rightarrow  \bW_h$ 
as in \cite{BF91,Bra01} such that  \[\int_{\Omega} \nabla \cdot \bPi_h \bu \,\mu_h \, d\bx = 
\int_{\Omega} \nabla \cdot  \bu\, \mu_h \, d\bx,\quad \mu_h \in  \tilde M^0_h.\]

Let $\Pi_h: H^1_0(\Omega) \rightarrow W_h$ be the scalar 
version of $\bPi_h$ meaning that  $\bPi_h \bu = (\Pi_h u_1,\Pi_h u_2)$ for 
the vector $\bu =(u_1,u_2)  \in  \bH^1_0(\Omega)$. 
Since $\bPi_h$ is bounded in $H^1$-norm, $\Pi_h$ is also 
bounded in $H^1$-norm. Since $\Pi_h \bu $ and $\bu$ both satisfy 
 homogeneous boundary condition, 
we even have 
 \[\int_{\Omega} \nabla \cdot \bPi_h \bu \,\mu_h \, d\bx = 
\int_{\Omega} \nabla \cdot  \bu\, \mu_h \, d\bx,\quad \mu_h \in  \tilde M_h.\]
If we use the function $\bu = (v,0)^T\in \bH^1_0(\Omega)$ in the above 
equation, we get 
\[ \int_{\Omega}
\left( \frac{\partial \Pi_h v}{\partial x} -\frac{\partial v}{\partial x}\right)\, \mu_h \, d\bx = 0,\quad \mu_h \in \tilde M_h,\]
and similarly for $ v \in H^1_0(\Omega)$, we have 
 \[ \int_{\Omega}
\left( \frac{\partial \Pi_h v}{\partial y} -\frac{\partial v}{\partial y}\right)\, \mu_h \, d\bx = 0,\quad \mu_h \in \tilde M_h.\]

Since $M_h^2 \subset \tilde M_h$, we have the following result. 
 \begin{lemma}\label{lemma2}
There exists a bounded linear projector $\Pi_h : H_0^1(\Omega) 
\rightarrow W_h$ such that for  $ v \in H^1_0(\Omega)$, we have 
 \[ \int_{\Omega}
\left( \frac{\partial \Pi_h v}{\partial x} -\frac{\partial v}{\partial x}\right)\, \mu_h \, d\bx = 0,\quad \mu_h \in M^2_h,\]
and 
 \[ \int_{\Omega}
\left( \frac{\partial \Pi_h v}{\partial y} -\frac{\partial v}{\partial y}\right)\, \mu_h \, d\bx = 0,\quad \mu_h \in  M^2_h.\]
 \end{lemma}
  
\begin{theorem}\label{th2}
The interpolation  operator $\Pi_h$ defined in Lemma \ref{lemma2}   satisfies
\[ \int_{\Omega} \nabla \Pi_h v\cdot \boeta_h\, d\bx  = 
\int_{\Omega} \nabla v\cdot \boeta_h\, d\bx,\;\boeta_h \in \bM^2_h.\]
\end{theorem}
\begin{proof}
Let  $\boeta_h = (\mu_h,\eta_h)^T \in \bM^2_h$, where 
$\mu_h, \eta_h \in M_h^2$. 
Then we need to satisfy 
\[ \int_{\Omega} \frac{\partial \Pi_h v}{\partial x} \, \mu_h\, d\bx + 
\int_{\Omega} \frac{\partial \Pi_h v}{\partial y} \, \eta_h\, d\bx 
  = \int_{\Omega} \frac{\partial v}{\partial x} \, \mu_h\, d\bx + 
\int_{\Omega} \frac{\partial  v}{\partial y} \, \eta_h\, d\bx ,\]
which results in adding two equations of Lemma \ref{lemma2}. 
\end{proof}
Thus Theorem \ref{th1} is proved with $R_h$ replaced by 
$\Pi_h$ . Hence we get the same approximation as in Theorem \ref{th0}.
\begin{theorem}\label{th3}
There exists a constant $c$ independent of $t$ and $h$ such that 
\begin{eqnarray*}
 \|\bphi-\bphi_h\|_{H^1(\Omega)} + \|u - u_h \|_{H^1(\Omega)} 
+ \||\bzeta -\bzeta_h\||_t  \leq \\
c \left( \inf_{\bpsi_h \in \bV_h} 
\|\bphi-\bpsi_h\|_{H^1(\Omega)} + \inf_{v_h \in W_h} 
\|u-v_h\|_{H^1(\Omega)}  +  \inf_{\boeta_h \in \bM^2_h} 
 \||\bzeta -\boeta_h\||_t  \right).
\end{eqnarray*}
Moreover, if $\bphi \in \bH^2(\Omega)$, 
$u \in H^2(\Omega)$ and $\bzeta\in \bH^1(\Omega)$, 
the approximation properties of $\bV_h$, $W_h$ and 
$\bM^2_h$ imply that 
\[ \|\bphi-\bphi_h\|_{H^1(\Omega)} + \|u - u_h \|_{H^1(\Omega)} 
+ \||\bzeta -\bzeta_h\||_t  \leq \\
C h \left(\|\bphi\|_{H^2(\Omega)} + \|u\|_{H^2(\Omega)}   + \|\bzeta\|_{H^1(\Omega)}\right).\]

\end{theorem}
\subsection{Algebraic formulation}    
We want to write the algebraic system of the 
discrete formulation \eqref{dsaddle}.
In the following, we use the same notation for the vector representation of 
the solutions and the solutions as elements in $\bV_h$, $W_h$ and 
$\bM_h$. The algebraic formulation of the saddle point problem 
\eqref{dsaddle} can be written as
\begin{equation} \label{saddlealg}
\left[\begin{array}{cccc} 
 \tA & \lambda \tB^T & \tD \\
\lambda \tB & \tK & \tB \\
\tD& \tB^T & \tR 
 \end{array} \right]
\left[\begin{array}{ccc} \bphi_h \\ u_h \\ \bzeta_h 
\end{array}\right]=
\left[\begin{array}{ccc} 0\\l_h\\0 \end{array}\right],
\end{equation}
where $\tA$, $\tB$, $\tK$, $\tD$ and $\tR$ are suitable matrices 
arising from the discretization of different bilinear forms, and 
$l_h$ is the vector form of discretization of the linear form 
$\ell(\cdot)$. Note that $\tD$ is the Gram matrix 
between basis functions  of $\bM_h$ and $\bV_h$. 
This matrix is a square matrix under Assumption 
\ref{A1A2}(i).  

We now briefly discuss the advantage of using 
the dual Lagrange multiplier space. 
Working with the dual Lagrange multiplier space  
$\tD$ will be a diagonal matrix due to 
the biorthogonality relation between the bases of 
$\bM^2_h$ and $\bV_h$.  The first equation of the 
algebraic system gives 
\[ \tA \bphi_h + \tB^T u_h + \tD \bzeta_h = 0.\]
This equation can be solved for $\bzeta_h$ as 
\[ \bzeta_h = - \tD^{-1} \left(\tA \bphi_h + \tB^T u_h \right).\]
 Thus we can statically condense out 
the Lagrange multiplier $\bzeta_h$ from the saddle point system. 
This leads to a reduced and positive definite system. Hence  an 
efficient solution technique can be applied to 
solve the arising linear system. 
\section{Conclusion}
We have combined the idea of constructing discrete Lagrange 
multiplier spaces in mortar finite element techniques 
to construct discrete Lagrange multiplier spaces for 
the Reissner--Mindlin plate equations. Working with 
a dual Lagrange multiplier space results in a very 
efficient finite element method.
\section*{Acknowledgement}
Support from the new staff grant of the University of Newcastle 
 is gratefully acknowledged.
\bibliographystyle{elsart-num-sort}
\bibliography{total}
\end{document}